\documentclass[10pt, twoside, fleqn, a4paper]{article}

\usepackage{latexsym}%  Extra math symbols
\usepackage{amscd}%     AMS package for doing commutative diagrams
\usepackage{theorem}%   Improved theoremlike environments
\usepackage{pifont}%    Dingbat symbols (used for end of proof symbol)
\usepackage{mathbbol}%  Blackboard bold symbols
\usepackage{amsfonts}%  More fonts
\usepackage{xspace}%    Removes need for adding a \  after macros
\usepackage{amssymb}%   More math symbols
\usepackage{fancyhdr}%  Fancyheadings package
\usepackage{mathrsfs}%  For the 3rd letters
\usepackage{amsmath}% Math symbols
\usepackage{graphics}% Picture inclusions
\usepackage{graphicx}% Picture inclusions
\usepackage{euscript}% 
\usepackage{psfrag}%
\usepackage{empheq}% Improved equation environments 
\usepackage{mathabx}% Includes a few more stylistic arrows
\usepackage[parfill]{parskip}    % Activate to begin paragraphs with an empty line rather than an indent
\usepackage[colorlinks=true, allcolors=blue]{hyperref}

\DeclareFontFamily{OT1}{pzc}{}
\DeclareFontShape{OT1}{pzc}{m}{it}%
              {<-> s * [0.900] pzcmi7t}{}
\DeclareMathAlphabet{\mathpzc}{OT1}{pzc}%
                                 {m}{it}

{\theorembodyfont{\slshape} \newtheorem{theorem}{Theorem}[section]}
{\theorembodyfont{\slshape} \newtheorem{definition}[theorem]{Definition}}
{\theorembodyfont{\slshape} \newtheorem{lemma}[theorem]{Lemma}}
{\theorembodyfont{\slshape} \newtheorem{proposition}[theorem]{Proposition}}
{\theorembodyfont{\slshape} }
{\theorembodyfont{\slshape} \newtheorem{remark}[theorem]{Remark}}
{\theorembodyfont{\slshape} \newtheorem{example}[theorem]{Example}}

\numberwithin{equation}{section}

\newenvironment{proof}{\paragraph{Proof:}}{\hfill$\bullet$}

\begin{document}

\title{Polynomial correspondences expressible as \\ maps of $d$-tuples}
\bigskip
\bigskip

\author{{\sc Subith Gopinathan}\footnote{{\sl Indian Institute of Science Education and Research Thiruvananthapuram (IISER-TVM), India.} \\ email: \texttt{subith21@iisertvm.ac.in}\ \ ORCID: 0009 0004 0041 8930 (Corresponding Author).}\ \ \ {\sc Shrihari Sridharan}\footnote{{\sl Indian Institute of Science Education and Research Thiruvananthapuram (IISER-TVM), India.} \\ email: \texttt{shrihari@iisertvm.ac.in}\ \ ORCID: 0000 0003 2434 4767.}\ \ \ {\sc Atma Ram Tiwari}\footnote{{\sl Rashtriya Postgraduate College, Veer Bahadur Singh Purvanchal University, Jaunpur, India.} \\ email: \texttt{art15iiser@gmail.com}\ \ ORCID: 0000-0002-7086-8986}}
\bigskip 

\date{\today}
\bigskip 

\maketitle

\begin{abstract}
\noindent 
In this paper, we consider polynomial correspondences $f (x, y)$ in $\mathbb{C}[x, y]$ of degree $d \ge 2$ in both the variables and obtain necessary and sufficient conditions in order that the equation $f (x, y) = 0$ can be expressed as $\Phi (x) = \Psi (y)$, where $\Phi$ and $\Psi$ are fractional degree $d$ rational maps in the Riemann sphere. In the absence of involutions that played a vital role towards characterising quadratic correspondences ($d = 2$), we employ certain elementary ideas and achieve our results. After exploring certain symmetry conditions on the matrix of coefficients of correspondences that satisfy the above factorisation, we conclude this short note with a few examples. 
\end{abstract}
\bigskip

\begin{tabular}{l c l}
\textbf{Keywords} & : & Polynomial correspondences, \\ & & Maps of $d$-tuples, \\ & & Equivalence classes in the Riemann sphere. \\\\ & & Matrix of coefficients \\ 
\textbf{AMS Subject Classifications} & : & 30D05, 30C10, 37F05, 37F10. \\ \\
\end{tabular}
\bigskip  

\thispagestyle{empty}

%\newpage 

\section{ Introduction}

The investigation of dynamics of polynomial correspondences that we carry out in this paper, primarily began with the works of Bullett in \cite{Bullett:1988}. Thereafter, Bullett wrote a series of articles with several co-authors, for example \cite{Bullett:1991, bf:05, Bullett:1986, Bullett:1994, bp:01}, in which the dynamics that arose out of iterating such correspondences were investigated. Various authors have also generalised the study of iteration of various types of correspondences, say, polynomial, rational, holomorphic, meromorphic, modular correspondence \textit{etc}. that are defined on the Riemann sphere or appropriately on certain manifolds. Apart from those mentioned above, some more notable works include Hinkkanen and Martin in \cite{hm:96a, hm:96b}, Boyd in \cite{db:99}, Sumi in \cite{hs:00, hs:01, hs:05}, Clozel and Otal in \cite{co:2001}, Mok in \cite{nm:2002}, Clozel and Ullmo in \cite{cu:2003}, Dinh in \cite{tcd:05:1, tcd:05:2}, Dinh and Sibony in \cite{ds:06}, Bharali and Sridharan in \cite{Bharali:2016, bs:17}, Dinh, Kaufmann and Wu in \cite{dkw:2020}. However, all these works focused on iterating the graph of the correspondence, by viewing it as a relation on the underlying space and allowing the dynamics to grow exponentially along each of the directions of the variables. 

The object of study in \cite{Bullett:1988} is the dynamical system that arises by iterating a certain relation on $ \widehat{\mathbb{C}}$, the Riemann sphere meaning the complex plane along with the point at infinity. This relation is the zero set of a polynomial $f \in \mathbb{C}[x, y]$ satisfying the following two conditions:
\begin{enumerate}
\item[1.] $f(x, \cdot)$ and $f(\cdot, y)$ are generically quadratic. 
\end{enumerate} 
Suppose for notational convenience and purposes of being explicit, we denote the domains of the first and the second variable in $f(x, y)$ by $\widehat{\mathbb{C}}_{x}$ and $\widehat{\mathbb{C}}_{y}$ respectively. Let $\Gamma_{f}$ denote the biprojective completion of the set $\left\{ f = 0 \right\}$ in $\widehat{\mathbb{C}}_{x} \times \widehat{\mathbb{C}}_{y}$. 
\begin{enumerate} 
\item[2.] No irreducible component of $\Gamma_{f}$ is of the form $\left\{ a \right\} \times \widehat{\mathbb{C}}_{y}$ or $\widehat{\mathbb{C}}_{x} \times \left\{ a \right\}$ for any $a \in \widehat{\mathbb{C}}$. 
\end{enumerate}

Observe that this can be well explained by considering a quartic correspondence in homogeneous coordinates $\left( [\xi : x], [\zeta : y] \right) \in \widehat{\mathbb{C}}_{x} \times \widehat{\mathbb{C}}_{y}$, where we use $\xi$ and $\zeta $ to homogenise each of the monomials, given by:  
\begin{eqnarray} 
\label{Homogenise,polynomial,for,d=2}
F \left( [\xi : x], [\zeta : y] \right) & := & A_{(2,\, 2)} x^{2} y^{2}\; +\; A_{(2,\, 1)} x^{2} y \zeta\; +\; A_{(1,\, 2)} x y^{2} \xi\; +\; A_{(2,\, 0)} x^{2} \zeta^{2}\; +\; A_{(0,\, 2)} y^{2} \xi^{2} \nonumber \\ 
& &  +\; A_{(1,\, 1)} x y \xi \zeta\; +\; A_{(1,\, 0)} x \xi \zeta^{2}\; +\; A_{(0,\, 1)} y \xi^{2} \zeta\; +\; A_{(0,\, 0)} \xi^{2} \zeta^{2}, 
\end{eqnarray}
where $A_{(i,\, j)} \in \mathbb{C}$. Then, $F([1 : x], [1 : y])$ yields a quadratic polynomial correspondence $f \in \mathbb{C}[x, y]$ that satisfies the above-mentioned two conditions. Suppose $\Pi_{x}$ and $\Pi_{y}$ are the projection maps of $\Gamma_{f}$ onto $\widehat{\mathbb{C}}_{x}$ and $\widehat{\mathbb{C}}_{y}$ respectively, then the conditions written above only imply that
\[ \widehat{\mathbb{C}}_{x} \ni x_{0} \longmapsto \Pi_{y} \left( \Pi_{x}^{-1} \left\{ x_{0} \right\} \cap \Gamma_{f} \right)\ \ \text{and}\ \ \widehat{\mathbb{C}}_{y} \ni y_{0} \longmapsto \Pi_{x} \left( \Pi_{y}^{-1} \left\{ y_{0} \right\} \cap \Gamma_{f} \right), \] 
are set-valued maps and have cardinality two for generic points $x_{0} \in \widehat{\mathbb{C}}_{x}$ and $y_{0} \in \widehat{\mathbb{C}}_{y}$. For such correspondences, Bullett, in \cite{Bullett:1988}, defined the following.

\begin{definition} 
\label{mop}
Let $x_{1} \in \widehat{\mathbb{C}}_{x}$. Then, $\Pi_{y} \left( \Pi_{x}^{-1} \left\{ x_{1} \right\} \cap \Gamma_{f} \right) = \left\{ y_{1}, y_{2} \right\}$, say. Further, $\Pi_{x} \left( \Pi_{y}^{-1} \left\{ y_{j} \right\} \cap \Gamma_{f} \right) = \left\{ x_{1}, x_{2} (j) \right\}$. Suppose $x_{2} (1) \equiv x_{2} (2)$ and the same is true for all initial points $x_{1} \in \widehat{\mathbb{C}}_{x}$. Moreover, we have an analogous situation even when we start with any $y_{1} \in \widehat{\mathbb{C}}_{y}$, then we call $f$ to be a \emph{map of pairs}.  
\end{definition} 

Polynomial correspondences constitute a good class of implicit functions to work with, since, ``Every holomorphic correspondence on the Riemann sphere is algebraic, that is to say, [it] can be expressed in the form $P \left( x, y \right) = 0$ for some polynomial $P$ in two variables", as proved by Bullett and Penrose, in \cite{bp:01}. For the quadratic polynomial correspondence $f \in \mathbb{C}[x, y]$ mentioned above, the author, in \cite{Bullett:1988}, has obtained characterising conditions for it to be a map of pairs. The process of tracing the orbits of points starting from an arbitrary point for a holomorphic correspondence is a tedious task, unlike in the case of maps. In such
a scenario, any means by which the underlying complexity can be simplified is of great help in understanding the dynamics systematically. 

Our aim in this paper is to consider a polynomial correspondence in $\mathbb{C} [x, y]$ of degree $d \ge 2$, along both the directions of the correspondence and investigate for characterising results for what one may call a map of $d$-tuples. A precise definition of this term, that is the central object of this paper, is given in Section \ref{two}. We also state the main theorems of this paper, namely Theorems \ref{Bullett,separation,for,d} and \ref{Bullett,for,d}, in the same section and prove them in Section \ref{pfsec1}. As one may observe from the statements of the main theorems, these are generalisations of the theorems due to Bullett in \cite{Bullett:1988}, where he discusses the case of $d = 2$. Moreover, the underlying complexity of polynomial correspondences is easier to handle when they turn out to be maps of $d$-tuples. Once the main theorems are proved, we look at certain interesting symmetries that are possible in the realm of the maps of $d$-tuples, in Section \ref{pfsec2}. We conclude the paper with a few examples that illustrate the theorems of this paper and the symmetries we would have by then discussed, in Section \ref{exsec}. 

%One important characterisation for $f$ to be a map of pairs is that the equation $f \left(x, y \right) = 0$ can be rewritten as $r(x) = s(y)$, where $r$ and $s$ are rational maps of degree $2$. As readers may observe, this immediately presents a simplification of the expression of the implicit function $f$ by separating the variables. 

\section{Maps of $d$-tuples and the Main results} 
\label{two} 

We begin this section by defining a degree $d$ polynomial correspondence denoted by $f_{d} \in \mathbb{C} [x, y]$ that satisfies the following conditions:
\begin{enumerate}
\item $f_{d} (x, \cdot) = 0$ and $f_{d} (\cdot, y) = 0$ are generically $d$-valued; 
\item The biprojective completion of the set $\left\{ f_{d} = 0 \right\}$ in $\widehat{\mathbb{C}}_{x} \times \widehat{\mathbb{C}}_{y}$ has no irreducible component of the form $\left\{ a \right\} \times \widehat{\mathbb{C}}_{y}$ or $\widehat{\mathbb{C}}_{x} \times \left\{ b \right\}$ for any $a \in \widehat{\mathbb{C}}$ or $b \in \widehat{\mathbb{C}}$. 
\end{enumerate}

Note that the significance of the second condition in our definition of the degree $d$ polynomial correspondence ensures that there exists no linear factor of the form $(x - a)$ or $(y - b)$ to $f_{d}$. The existence of such a linear factor leads us to the cardinality of the set $f_{d} (a, \cdot) = 0$ or $f_{d} (\cdot, b) = 0$ to be uncountable and the degree of the correspondence $f_{d}$ being smaller than $d$ for that variable. 

Observe that such a polynomial correspondence is then given by 
\begin{eqnarray} 
\label{fdnozeta} 
f_{d}(x, y) & := & A_{(d,\, d)} x^{d} y^{d}\; +\; A_{(d,\, d - 1)} x^{d} y^{d - 1}\; +\; \cdots\; +\; A_{(d,\, 1)} x^{d} y\; +\; A_{(d,\, 0)} x^{d} \nonumber \\ 
& & +\; A_{(d - 1,\, d)} x^{d - 1} y^{d}\; +\; \cdots\; +\; A_{(d - 1,\, 0)} x^{d - 1} \nonumber \\ 
& & +\; \cdots \nonumber \\ 
& & +\; A_{(1,\, d)} x y^{d}\; +\; \cdots\; +\; A_{(1,\, 0)} x \nonumber \\ 
& & +\; A_{(0,\, d)} y^{d}\; +\; \cdots\; +\; A_{(0,\, 0)}.  
\end{eqnarray} 

Observe that for any generic $x_{0} \in \widehat{\mathbb{C}}_{x}$ and $y_{0} \in \widehat{\mathbb{C}}_{y},\ f_{d} \in \mathbb{C}[x, y]$ satisfies the following:
\[ \widehat{\mathbb{C}}_{x} \ni x_{0} \longmapsto \Pi_{y} \left( \Pi_{x}^{-1} \left\{ x_{0} \right\} \cap \Gamma_{f_{d}} \right)\ \ \text{and}\ \ \widehat{\mathbb{C}}_{y} \ni y_{0} \longmapsto \Pi_{x} \left( \Pi_{y}^{-1} \left\{ y_{0} \right\} \cap \Gamma_{f_{d}} \right) \] 
are set-valued maps with sets of cardinality $d$. Analogous as earlier, $\Gamma_{f_{d}}$ represents the biprojective completion of the set $\left\{ f_{d} = 0 \right\}$ in $\widehat{\mathbb{C}}_{x} \times \widehat{\mathbb{C}}_{y}$. We now consider an example to illustrate the central theme of this paper, namely a map of $d$-tuples and follow it up with the formal definition. 

\begin{example} 
Consider $f_{d}(x, y) = x^{d} - (2y)^{d}$. Suppose $\omega$ is a primitive $d$-th root of unity in $\mathbb{C}$. Then for any $a \in \widehat{\mathbb{C}}_{x}$, the equation $f_{d} (a, y) = 0$ has the solution set given by $\displaystyle{\left\{ \frac{a}{2}, \frac{a \omega}{2}, \cdots, \frac{a \omega^{d - 1}}{2} \right\}}$. Moreover, equation $\displaystyle{f_{d} \left( x, \frac{a \omega^{j}}{2} \right) = 0}$ gives the same set of solutions $\left\{ a, a \omega, \cdots, a \omega^{d - 1} \right\}$ for any $0 \le j \le d - 1$. In fact, the equation $f_{d} (a \omega^{k}, y) = 0$ yields the same solution set $\displaystyle{\left\{ \frac{a}{2}, \frac{a \omega}{2}, \cdots, \frac{a \omega^{d - 1}}{2} \right\}}$, for any $0 \le k \le d - 1$. 
\end{example}

In general, this need not be the case. Starting with the solution set of $f_{d} (a, y) = 0$ for any $a \in \widehat{\mathbb{C}}_{x}$, we obtain $d$ points (counting multiplicities) in $\widehat{\mathbb{C}}_{y}$, and plugging each of these values for the variable $y$ in $f_{d} (x, y) = 0$ yields us a set of cardinality $d^{2}$ (counting multiplicities) in $\widehat{\mathbb{C}}_{x}$. The whole process is reversible too, when we start with an arbitrary $b \in \widehat{\mathbb{C}}_{y}$. Hence, it is interesting to study polynomial correspondences where the cardinality can be controlled, as illustrated in the above example. Towards that end is the following definition. 

\begin{definition} 
\label{modtuple} 
For each $x_{1} \in \widehat{\mathbb{C}}_{x}$, let $\left\{ y_{j}(x_{1}) \right\}_{1\, \le\, j\, \le\, d}$ denote the points in $\Pi_{y} \left( \Pi_{x}^{-1} \left\{ x_{1} \right\} \cap \Gamma_{f_{d}} \right)$ repeated according to the intersection multiplicities of $\Pi_{x}^{-1} \left\{ x_{1} \right\} \cap \Gamma_{f_{d}}$ at its intersection points (equivalently: according to the multiplicity of the zeros of $f_{d} (x_{1}, \cdot)$ when $x_{1} \ne \infty$) and let 
\[ \Big\{ x_{1},\; x_{2} (y_{j} (x_{1})),\; \cdots,\; x_{d} (y_{j} (x_{1})) \Big\},\ \ \ j = 1, \cdots, d, \] 
denote the points in $\Pi_{x} \left( \Pi_{y}^{-1} \left\{ y_{j} (x_{1}) \right\} \cap \Gamma_{f_{d}} \right)$ repeated according to intersection multiplicity. For each $y_{1} \in \widehat{\mathbb{C}}_{y}$, considering $\Pi_{x} \left( \Pi_{y}^{-1} \left\{ y_{1} \right\} \cap \Gamma_{f_{d}} \right)$, let 
\[ \Big\{ y_{1},\; y_{2} (x_{j} (y_{1})),\; \cdots,\; y_{d} (x_{j} (y_{1})) \Big\},\ \ \ j = 1, \cdots, d \] 
be constructed as above \emph{mutatis mutandis}. We call $f_{d}$ a map of $d$-tuples, which we denote by 
\[ \big( x_{1}, x_{2}, \cdots, x_{d} \big)_{\widehat{\mathbb{C}}_{x}}\ \ \mathrel{\mathop{\rightleftarrows}^{f_{d}}}\ \ \big( y_{1}, y_{2}, \cdots, y_{d} \big)_{\widehat{\mathbb{C}}_{y}} \] 
if for each $x_{1} \in \widehat{\mathbb{C}}_{x}$, the lists $\big\{ x_{1}, x_{2} (y_{j} (x_{1})), \cdots, x_{d} (y_{j} (x_{1})) \big\}$ for $j = 1, \cdots, d$ coincide and for each $y_{1} \in \widehat{\mathbb{C}}_{y}$, the lists $\big\{ y_{1}, y_{2} (x_{j} (y_{1})), \cdots, y_{d} (x_{j} (y_{1})) \big\}$ for $j = 1, \cdots, d$ coincide. 
\end{definition} 

\begin{remark} 
\label{newrem} 
It is clear from the definition that if $f_{d} \left( x,\, y \right)$ is a map of $d$-tuples with 
\[ \big( x_{1}, x_{2}, \cdots, x_{d} \big)_{\widehat{\mathbb{C}}_{x}}\ \ \mathrel{\mathop{\rightleftarrows}^{f_{d}}}\ \ \big( y_{1}, y_{2}, \cdots, y_{d} \big)_{\widehat{\mathbb{C}}_{y}} \] 
then, $f_{d} \left( x_{i},\, y_{j} \right) = 0$ for every $x_{i} \in \left\{ x_{1}, x_{2}, \cdots, x_{d} \right\}$ and for every $y_{j} \in \left\{ y_{1}, y_{2}, \cdots, y_{d} \right\}$. Moreover, for any $1 \le j \le d$, we have $f_{d} \left( x,\, y_{j} \right) \ne 0$ unless $x \in \left\{ x_{1}, x_{2}, \cdots, x_{d} \right\}$ and for $ 1 \le i \le d$, we have $f_{d} \left( x_{i},\, y \right) \ne 0$ unless $y \in \left\{ y_{1}, y_{2}, \cdots, y_{d} \right\}$. 
\end{remark} 

\begin{remark} 
\label{rmodtuple} 
For ease of notations and convenience, we assume that the list of points $x_{i}$'s and $y_{j}$'s , as mentioned in the above definition are ordered by non-decreasing modulus, followed by the order of non-decreasing arguments in the interval $[0, 2\pi)$. 
\end{remark}

We now state the two main theorems of this paper. 
 
\begin{theorem} 
\label{Bullett,separation,for,d} 
Suppose $f_{d} (x, y)$, as written in Equation \eqref{fdnozeta}, can not be written as $\displaystyle{f_{d} (x, y) = \big[ g_{k} (x, y) \big]^{\frac{d}{k}}}$ where $k | d$ and $k \ne d$ with $g_{k} \in \mathbb{C}[x, y]$ being a map of $k$-tuples. Then, the correspondence $f_{d}$ is a map of $d$-tuples if and only if the variables $x$ and $y$ can be separated in the equation $f_{d} (x, y)\ =\ 0$, {\it i.e.}, the equation $f_{d} (x, y)\ =\ 0$ can be rewritten as $\Phi (x)\ =\ \Psi (y)$, where $\Phi$ and $\Psi$ are fractional degree $d$ functions, as given below:  
\begin{equation} 
\label{phipsi} 
\Phi (x)\ =\ \frac{\kappa_{d} x^{d} + \cdots + \kappa_{1} x + \kappa_{0}}{\lambda_{d} x^{d} + \cdots + \lambda_{1} x + \lambda_{0}} \quad \quad \text{and} \quad \quad \Psi (y)\ =\ \frac{\mu_{d} y^{d} + \cdots + \mu_{1} y + \mu_{0}}{\nu_{d} y^{d} + \cdots + \nu_{1} y + \nu_{0}}, 
\end{equation}  
with complex coefficients.  
\end{theorem}

\begin{theorem} 
\label{Bullett,for,d}
The correspondence $f_{d}$, as written in Equation \eqref{fdnozeta} is a map of $d$-tuples on the Riemann sphere if and only if one of the following conditions is true: 
\begin{enumerate} 
\item $f_{d} (x, y)$ is of the form $\displaystyle{f_{d} (x, y) = \big[ g_{k} (x, y) \big]^{\frac{d}{k}}}$ where $k | d$ and $k \ne d$ with $g_{k} \in \mathbb{C}[x, y]$ being a map of $k$-tuples. 
\item The matrix of coefficients of $f_{d}$ has rank $2$, {\it i.e.},  
\begin{equation} 
\label{Main,matrix,for,d}
{\rm rank} \begin{pmatrix} 
A_{(d,\, d)} & A_{(d,\, d - 1)} & \cdots & A_{(d,\, 0)} \\ 
A_{(d - 1,\, d)} & A_{(d - 1,\, d - 1)} & \cdots  & A_{(d - 1,\, 0)} \\ 
\vdots & \vdots & & \vdots \\ 
A_{(0,\, d)} & A_{(0,\, d - 1)} & \cdots & A_{(0,\, 0)} 
\end{pmatrix}_{(d + 1)\, \times\, (d + 1)}\ \ =\ \ 2. 
\end{equation}
\end{enumerate} 
\end{theorem}

We conclude this section with the following remarks. 

\begin{remark} 
Upon taking $d = 2$ in Theorems \ref{Bullett,separation,for,d} and \ref{Bullett,for,d}, we obtain Bullett's results found in \cite{Bullett:1988}, as corollaries to the respective statements of our theorems.  
\end{remark} 

\begin{remark} 
In the case of a map of pairs, as explained in Definition \ref{mop}, we urge the readers to observe the presence of involutions $I_{x} : \widehat{\mathbb{C}}_{x} \righttoleftarrow$ and $I_{y} : \widehat{\mathbb{C}}_{y} \righttoleftarrow$ that satisfies $I_{x} (x_{1}) = x_{2}$ and $I_{y} (y_{1}) = y_{2}$. Bullett, in \cite{Bullett:1988} makes an extensive use of these involutions to prove his results. However when $d > 2$ (the case that we deal with), we do not have such involutions. Thus, we employ alternative methods to prove our theorems.  
\end{remark} 

 \section{Proofs of Theorem \ref{Bullett,separation,for,d} and Theorem \ref{Bullett,for,d}}
\label{pfsec1} 

In this section of the manuscript, we prove the two main theorems. We first prove Theorem \ref{Bullett,separation,for,d}. 

\begin{proof}(of Theorem \ref{Bullett,separation,for,d}) We start by supposing $f_{d}$ to be a map of $d$-tuples that can not be written as $\big[ g_{k} (x, y) \big]^{\frac{d}{k}}$ where $k | d$ and $k \ne d$ with $g_{k} \in \mathbb{C}[x, y]$ being a map of $k$-tuples. Then, for any $x_{1}$ and $x_{2}$ in $\widehat{\mathbb{C}}_{x}$, we say  $x_{1} \sim x_{2}$ if and ony if the list of points in $\Pi_{y} \left( \Pi_{x}^{-1} \{ x_{1} \} \cap \Gamma_{f_{d}} \right)$ and $\Pi_{y} \left( \Pi_{x}^{-1} \{ x_{2} \} \cap \Gamma_{f_{d}} \right)$, repeated according to multiplicity are equal. Analogously, for any $y_{1}$ and $y_{2}$ in $\widehat{\mathbb{C}}_{y}$, we say  $y_{1} \sim y_{2}$ if and ony if the list of points in $\Pi_{x} \left( \Pi_{y}^{-1} \{ y_{1} \} \cap \Gamma_{f_{d}} \right)$ and $\Pi_{x} \left( \Pi_{y}^{-1} \{ y_{2} \} \cap \Gamma_{f_{d}} \right)$, repeated according to multiplicity are equal.  It is easy to verify that $\sim$ is an equivalence relation on $\widehat{\mathbb{C}}_{x}$, as well as on $\widehat{\mathbb{C}}_{y}$. Denote by $\left[ x_{1} \right]$, the equivalence class of elements in $\widehat{\mathbb{C}}_{x}$ that are related to $x_{1}$ and by $\left[ y_{1} \right]$, the equivalence class of elements in $\widehat{\mathbb{C}}_{y}$ that are related to $y_{1}$. 

Since $f_{d}$ is a map of $d$-tuples, from Definition \ref{modtuple}, we have for every $x_{1} \in \widehat{\mathbb{C}}_{x}$, 
\[ \big( x_{1}, x_{2}, \cdots, x_{d} \big)_{\widehat{\mathbb{C}}_{x}}\ \ \mathrel{\mathop{\rightleftarrows}^{f_{d}}}\ \ \big( y_{1}, y_{2}, \cdots, y_{d} \big)_{\widehat{\mathbb{C}}_{y}},\ \ \ \ \text{{\it i.e.}},\ f_{d} \left( x_{i},\, y_{j} \right) = 0,\ \ \text{for}\ 1 \le i, j \le d. \] 
In other words, the solutions of the equation $f_{d} \left( x_{i}, y \right) = 0$ for every $1 \le i \le d$ are precisely the list of points $\left\{ y_{1},\, \cdots,\, y_{d} \right\}$. Thus, $x_{i} \in \left[ x_{1} \right]$ for every $1 \le i \le d$. Suppose there exists $x^{*} \in \left[ x_{1} \right]$ with $x^{*} \ne x_{i}$ for any of these $1 \le i \le d$, then the solutions of the equation $f_{d} \left( x^{*}, y \right) = 0$ is also equal to the list of points $\left\{ y_{1},\, \cdots,\, y_{d} \right\}$. However, this is in violation to the fundamental theorem of algebra, since there are more than $d$ points, counted with multiplicities, namely, $\left\{ x_{1},\, \cdots, x_{d},\, x^{*} \right\}$ that satisfies the equation $f_{d} \left( x, y_{j} \right) = 0$ for any $1 \le j \le d$. Hence, every equivalence class in $\widehat{\mathbb{C}}_{x}$ and $\widehat{\mathbb{C}}_{y}$ contains $d$ points, counting multiplicity. 

\noindent 
Suppose we construct a map $\phi : \widehat{\mathbb{C}}_{x} \longrightarrow \widehat{\mathbb{C}}$ such that $\phi (x_{1}) = \phi (x_{2})$ if and ony if $x_{1} \sim x_{2}$ and a map $\psi : \widehat{\mathbb{C}}_{y} \longrightarrow \widehat{\mathbb{C}}$ such that $\psi (y_{1}) = \psi (y_{2})$ if and ony if $y_{1} \sim y_{2}$. If such a construction is possible, then $\phi$ assigns the same value to all the $d$ points (counting multiplicity) that are related to each other in $\widehat{\mathbb{C}}_{x}$ and hence, must be a rational map of degree $d$. Similarly, $\psi$ must also be a rational map of degree $d$. We can then employ an automorphism, namely a M\"{o}bius map, to move within $\widehat{\mathbb{C}}$. It is essential to understand that we employ a M\"{o}bius map to move within $\widehat{\mathbb{C}}$, since $f_{d}$ is only a map of $d$-tuples and we should not accrue any extra degrees at this stage. Thus, the map that we employ needs to be linear. Further, this map should also be an analytic bijection. In other words, we concern ourselves with factorising $f_{d}$ as: 
\begin{equation} 
\label{likeBullett} 
\widehat{\mathbb{C}}_{x}\ \ \overset{\phi}\longrightarrow\ \ \widehat{\mathbb{C}}\ \ \overset{M}\longrightarrow\ \ \widehat{\mathbb{C}}\ \ \overset{\psi}\longleftarrow\ \  \widehat{\mathbb{C}}_{y}. 
\end{equation} 
Then, the equation $f_{d} (x, y) = 0$ can be re-written as $M \circ \phi (x) = \psi (y)$. Towards that end, we re-write the equation $f_{d} (x, y) = 0$ as 
\begin{eqnarray} 
\label{polynomialiny} 
P_{d} (x) y^{d} + P_{d - 1} (x) y^{d - 1} + \cdots + P_{1} (x) y + P_{0} (x) & = & 0\ \ \ \text{and} \\ 
\label{polynomialinx} 
Q_{d} (y) x^{d} + Q_{d - 1} (y) x^{d - 1} + \cdots + Q_{1} (y) x + Q_{0} (y) & = & 0,  
\end{eqnarray} 
where each $P_{j}$ in Equation \eqref{polynomialiny} and each $Q_{j}$ in Equation \eqref{polynomialinx} for $0 \le j \le d$ is a polynomial in the variable $x$ and $y$ respectively, of degree at most $d$. Also at least one of the polynomials $P_{j}$ for $0 \le j \le d$ and one of the polynomials $Q_{j}$ for $0 \le j \le d$ is necessarily of degree $d$. Without loss of generality (see Remark \ref{jatatime} after this proof), let ${\rm deg} (P_{d}) = d$. Further, one can not have all the polynomials $P_{j}$ for $0 \le j \le d - 1$ to be some constant multiple of $P_{d}$ simultaneously. Thus, without loss of generality (see Remark \ref{jatatime} after this proof), let $P_{d - 1}$ be the polynomial that can not be written as $c P_{d}$ for any $c \in \mathbb{C}$. 

Since $f_{d}$ is a map of $d$-tuples, one can see that whenever $\displaystyle{\big( x_{1}, x_{2}, \cdots, x_{d} \big)_{\widehat{\mathbb{C}}_{x}}\ \mathrel{\mathop{\rightleftarrows}^{f_{d}}}\ \big( y_{1}, y_{2}, \cdots, y_{d} \big)_{\widehat{\mathbb{C}}_{y}}}$, we have $\left[ x_{i} \right] = \big\{ x_{1}, x_{2}, \cdots, x_{d} \big\}$ and $\left[ y_{j} \right] =\big\{ y_{1}, y_{2}, \cdots, y_{d} \big\}$ for $1 \le i, j \le d$. Consider an arbitrary pair of $d$-tuples that are related to each other by the corresponcence $f_{d}$, say $\displaystyle{\big( a_{1}, a_{2}, \cdots, a_{d} \big)_{\widehat{\mathbb{C}}_{x}} \mathrel{\mathop{\rightleftarrows}^{f_{d}}} \big( b_{1}, b_{2}, \cdots, b_{d} \big)_{\widehat{\mathbb{C}}_{y}}}$. Then, observe that the points $\left\{ a_{1}, a_{2}, \cdots, a_{d} \right\}$ are the $d$ solutions to the equation $\displaystyle{\dfrac{P_{d - 1}}{P_{d}} (x) = - \sum_{b\, \in\, \left[ b_{1} \right]} b}$, from Equation \eqref{polynomialiny}. Similarly, with analogous assumptions that we considered without loss of generalities, we obtain $\left\{ b_{1}, b_{2}, \cdots, b_{d} \right\}$ to be the $d$ solutions to the equation $\displaystyle{\dfrac{Q_{d - 1}}{Q_{d}} (y) = - \sum_{a\, \in\, \left[ a_{1} \right]} a}$ from Equation \eqref{polynomialinx}. In fact, under the assumptions, that we made without loss of generalities, for any pair of $d$-tuples that satisfy $\displaystyle{\big( x_{1}, x_{2}, \cdots, x_{d} \big)_{\widehat{\mathbb{C}}_{x}} \mathrel{\mathop{\rightleftarrows}^{f_{d}}} \big( y_{1}, y_{2}, \cdots, y_{d} \big)_{\widehat{\mathbb{C}}_{y}}}$, it is clear from Equations \eqref{polynomialiny} and \eqref{polynomialinx} that $\displaystyle{\dfrac{P_{d - 1}}{P_{d}} (x_{i}) = - \sum_{y\, \in\, \left[ y_{1} \right]} y}$ for any $x_{i} \in \left\{ x_{1}, x_{2}, \cdots, x_{d} \right\}$ and $\displaystyle{\dfrac{Q_{d - 1}}{Q_{d}} (y_{j}) = - \sum_{x\, \in\, \left[ x_{1} \right]} x}$ for any $y_{j} \in \left\{ y_{1}, y_{2}, \cdots, y_{d} \right\}$. Hence, $\displaystyle{\dfrac{P_{d - 1}}{P_{d}}}$ and $\displaystyle{\dfrac{Q_{d - 1}}{Q_{d}}}$ (under the assumptions, that we made without loss of generalities) are well-defined maps (upto a constant multiple) that takes $d$ many points in $\widehat{\mathbb{C}}_{x}$ to a single point in $\widehat{\mathbb{C}}$ and $d$ many points in $\widehat{\mathbb{C}}_{y}$ to a single point in $\widehat{\mathbb{C}}$, respectively. Thus, it is necessary for $\displaystyle{\dfrac{P_{d - 1}}{P_{d}}}$ and $\displaystyle{\dfrac{Q_{d - 1}}{Q_{d}}}$ to be rational maps of degree $d$, thereby ruling out the possibility of the numerator and the denominator in the respective rational maps sharing any common factors. Hence, we consider $\displaystyle{\phi (x) = \dfrac{P_{d - 1}}{P_{d}} (x)}$ and $\displaystyle{\psi (y) = \dfrac{Q_{d - 1}}{Q_{d}} (y)}$. Now, by considering an appropriate M\"{o}bius map, say $M$ to move within $\widehat{\mathbb{C}}$ and defining $\Phi = M \circ \phi$ and $\Psi \equiv \psi$, we obtain 
\[ \Phi (x)\ \ =\ \ \frac{\kappa_{d} x^{d}\, +\, \cdots\, +\, \kappa_{1} x\, +\, \kappa_{0}}{\lambda_{d} x^{d}\, +\, \cdots\, +\, \lambda_{1} x\, +\, \lambda_{0}} \quad \quad \text{and} \quad \quad \Psi (y)\ \ =\ \ \frac{\mu_{d} y^{d}\, +\, \cdots\, +\, \mu_{1} y\, +\, \mu_{0}}{\nu_{d} y^{d}\, +\, \cdots\, +\, \nu_{1} y\, +\, \nu_{0}}, \] 
where $\kappa_{i}, \lambda_{i}, \mu_{i}, \nu_{i} \in \mathbb{C}$ for $0 \le i \le d$. 

Conversely, suppose the equation $f_{d} (x, y) = 0$ can be re-written by separating the variables as $\Phi (x) = \Psi (y)$ where $\Phi$ and $\Psi$ are fractional degree $d$ functions. In order to prove that $f_{d}$ is a map of $d$-tuples, we observe that for any complex value $z_{0} \in \widehat{\mathbb{C}}$, both the equations $\Phi (x) = z_{0}$ and $\Psi (y) = z_{0}$ have solution sets, each of cardinality $d$ counting multiplicity, thus making $f_{d}$, a map of $d$-tuples. 
\end{proof} 

\begin{remark} 
\label{jatatime} 
Here, we explain why our choices of the polynomials $P_{d}$ and $P_{d - 1}$ in the proof of Theorem \ref{Bullett,separation,for,d} are without loss of generalities. Suppose we replace our choices of $P_{d}$ and $P_{d - 1}$ with $P_{d - j}$ and $P_{d - k}$ respectively. For any pair of $d$-tuples that obeys $\displaystyle{\big( x_{1}, x_{2}, \cdots, x_{d} \big)_{\widehat{\mathbb{C}}_{x}} \mathrel{\mathop{\rightleftarrows}^{f_{d}}} \big( y_{1}, y_{2}, \cdots, y_{d} \big)_{\widehat{\mathbb{C}}_{y}}}$, the rational function $\displaystyle{\frac{P_{d - k}}{P_{d - j}} = \frac{P_{d - k}}{P_{d}}\frac{P_{d}}{P_{d - j}}}$ takes the $d$ points $x_{1}, x_{2}, \cdots, x_{d}$, counting multiplicity to the value $\displaystyle{\frac{(-1)^{k} \sum\limits_{1\, \le\, i_{1}\, <\, \cdots\, <\, i_{k}\, \le d} \left( y_{i_{1}} \times \cdots \times y_{i_{k}} \right)}{(-1)^{j} \sum\limits_{1\, \le\, l_{1}\, <\, \cdots\, <\, l_{j}\, \le d} \left( y_{l_{1}} \times \cdots \times y_{l_{j}} \right)}}$, and thus, the set of points $\left\{ x_{1}, x_{2}, \cdots, x_{d} \right\}$, repeated according to multiplicities are the only solutions to the equation 
\[ \frac{P_{d - k}}{P_{d - j}} (x)\ \ =\ \ \frac{(-1)^{k} \sum\limits_{1\, \le\, i_{1}\, <\, \cdots\, <\, i_{k}\, \le d} \left( y_{i_{1}} \times \cdots \times y_{i_{k}} \right)}{(-1)^{j} \sum\limits_{1\, \le\, l_{1}\, <\, \cdots\, <\, l_{j}\, \le d} \left( y_{l_{1}} \times \cdots \times y_{l_{j}} \right)}. \] 
An analogous argument is true for the polynomials $Q_{j}$'s as well. 
\end{remark} 

As one may observe from the proof of Theorem \ref{Bullett,separation,for,d}, the factorisation of $f_{d} (x, y) = 0$ as $\Phi (x) = \Psi (y)$ is unique only upto composition with a M\"{o}bius map. We now prove Theorem \ref{Bullett,for,d}. 

\begin{proof}(of Theorem \ref{Bullett,for,d}) Suppose $f_{d} =\big[ g_{k} (x, y) \big]^{\frac{d}{k}}$ where $k | d$ and $k \ne d$ with $g_{k} \in \mathbb{C}[x, y]$ being a map of $k$-tuples. Then, it is obvious that $f_{d}$ is a map of $d-$tuples. 

In order to prove the second statement of Theorem \ref{Bullett,for,d}, we make use of Theorem \ref{Bullett,separation,for,d}. Suppose $f_{d}$ cannot be written as $\big[ g_{k} (x, y) \big]^{\frac{d}{k}}$ where $k | d$ and $k \ne d$ with $g_{k}$ being a map of $k$-tuples. Then, $f_{d} (x, y) = 0$ can be rewritten as $\Phi (x) = \Psi (y)$,  where 
\[ \Phi (x)\ \ =\ \ \frac{\kappa_{d} x^{d}\, +\, \cdots\, +\, \kappa_{1} x\, +\, \kappa_{0}}{\lambda_{d} x^{d}\, +\, \cdots\, +\, \lambda_{1} x\, +\, \lambda_{0}} \quad \quad \text{and} \quad \quad \Psi (y)\ \ =\ \ \frac{\mu_{d} y^{d}\, +\, \cdots\, +\, \mu_{1} y\, +\, \mu_{0}}{\nu_{d} y^{d}\, +\, \cdots\, +\, \nu_{1} y\, +\, \nu_{0}}, \] 
with $\kappa_{i}, \lambda_{i}, \mu_{i}, \nu_{i} \in \mathbb{C}$ for $0 \le i \le d$. Thus, in this case the correspondence can be re-written as 
\[ \left( \kappa_{d} \nu_{d} - \lambda_{d} \mu_{d} \right) x^{d} y^{d}\; +\; \left( \kappa_{d} \nu_{d - 1} - \lambda_{d} \mu_{d - 1} \right) x^{d} y^{d - 1}\; +\; \cdots\; +\; \left( \kappa_{0} \nu_{0} - \lambda_{0} \mu_{0} \right)\ \ =\ \ 0. \] 
Thus, the appropriate matrix of coefficients is given by
\[ \begin{pmatrix} \kappa_{d} \nu_{d} - \lambda_{d} \mu_{d} & \kappa_{d} \nu_{d - 1} - \lambda_{d} \mu_{d - 1} & \cdots & \kappa_{d} \nu_{0} - \lambda_{d} \mu_{0} \\ 
\kappa_{d - 1} \nu_{d} - \lambda_{d - 1} \mu_{d} & \kappa_{d - 1} \nu_{d - 1} - \lambda_{d - 1} \mu_{d - 1} & \cdots & \kappa_{d - 1} \nu_{0} - \lambda_{d - 1} \mu_{0} \\ 
\vdots & & & \\ 
\kappa_{0} \nu_{d} - \lambda_{0} \mu_{d} & \kappa_{0} \nu_{d - 1} - \lambda_{0} \mu_{d - 1} & \cdots & \kappa_{0} \nu_{0} - \lambda_{0} \mu_{0}, \end{pmatrix} \] 
that can be factorised as 
\[ \begin{pmatrix} \kappa_{d} & 0 & \cdots & 0 & \lambda_{d} \\ 
\kappa_{d - 1} & 0 & \cdots & 0 & \lambda_{d - 1} \\ 
\vdots & & & &\vdots \\ 
\kappa_{0} & 0 & \cdots & 0 & \lambda_{0} \end{pmatrix} \times 
\begin{pmatrix} \nu_{d} & \nu_{d - 1} & \cdots & \nu_{0} \\ 
0 & 0 & \cdots & 0 \\ 
\vdots & & & & \\ 
0 & 0 & \cdots & 0 \\ 
- \mu_{d} & - \mu_{d - 1} & \cdots & - \mu_{0} \end{pmatrix}. \] 
Since both the matrices in the factorisation have their ranks to be two, the product matrix can have rank utmost two. We now prove that the rank of the product matrix is not equal to one, by the method of contradiction. 

Consider the matrix of coefficients given by 
\[ \begin{pmatrix} 
A_{(d,\, d)} & A_{(d,\, d - 1)} & \cdots & A_{(d,\, 0)} \\ 
A_{(d - 1,\, d)} & A_{(d - 1,\, d - 1)} & \cdots  & A_{(d - 1,\, 0)} \\ 
\vdots & \vdots & & \vdots \\ 
A_{(0,\, d)} & A_{(0,\, d - 1)} & \cdots & A_{(0,\, 0)} 
\end{pmatrix}. \] 
Since, by assumption, this matrix has rank $1$, we write the latter $(d - 1)$ columns, (without loss of generality) as a constant multiple of the first column. Then, the equation of the correspondence $f_{d} (x, y)= 0$ can be factorised as $P(x) Q(y) = 0$, where $P$ and $Q$ are polynomials of degree $d$ with complex coefficients. This violates the second condition in the definition of the degree $d$ polynomial correspondence.  

We now suppose that the rank of the coefficient matrix is $2$ and prove that $f_{d}$ is a map of $d$-tuples. Denote the rows of the coefficient matrix by $R_{1}, R_{2}, \cdots, R_{d + 1}$. Suppose that the rows $R_{k}$ for $1 \le k \le d + 1$ are spanned by the row $R_{i}$ and the row $R_{j}$, {\it i.e.}, $R_{k} = \sigma_{i} R_{i} + \tau_{j} R_{j}$. Thus, one can write $f_{d} (x, y) = 0$ as 
\[ x^{d} R_{1} Y\; +\; x^{d - 1} R_{2} Y\; +\; \cdots\; +\; x R_{d} Y\; +\; R_{d + 1} Y\ \ =\ \ 0,\ \ \ \ \text{where}\ \ \ Y^{t}\ =\ \begin{pmatrix} y^{d} & y^{d - 1} & \cdots & y & 1 \end{pmatrix}. \] 
This can be simplified as 
\[ \left( \sigma_{1} x^{d} + \sigma_{2} x^{d - 1} + \cdots + \sigma_{d} x + \sigma_{d + 1} \right) R_{i} Y\ +\ \left( \tau_{1} x^{d} + \tau_{2} x^{d - 1} + \cdots + \tau_{d} x + \tau_{d + 1} \right) R_{j} Y\ \ =\ \ 0. \] 
which can be rewritten as 
\begin{equation} 
\label{label} 
\frac{\sigma_{1} x^{d} + \sigma_{2} x^{d - 1} + \cdots + \sigma_{d} x + \sigma_{d + 1}}{\tau_{1} x^{d} + \tau_{2} x^{d - 1} + \cdots + \tau_{d} x + \tau_{d + 1}}\ \ =\ \ - \frac{A_{(j - 1,\, d)} y^{d} + A_{(j - 1,\, d - 1)} y^{d - 1} \cdots + A_{(j - 1,\, 0)}}{A_{(i - 1,\, d)} y^{d} + A_{(i - 1,\, d - 1)} y^{d - 1} + \cdots + A_{(i - 1,\, 0)}}. 
\end{equation} 
Since this equation is a restatement of $f_{d} (x, y) = 0$, both $\sigma_{1}$ and $\tau_{1}$ can not be simultaneously zero. We define the expression in the LHS of Equation \eqref{label} as $\Phi (x)$ and the expression in the RHS of Equation \eqref{label} as $\Psi (y)$. If the numerator and denominator of $\Phi$ share a common factor, then it violates the condition that $f_{d} (x,\, y)$ does not have a factor $(x - a)$ for any $a \in \mathbb{C}$. The same is true for $\Psi$ too. Thus, we have $\Phi (x) = \Psi (y)$, where $\Phi$ and $\Psi$ are fractional degree $d$ maps. An appeal to Theorem \ref{Bullett,separation,for,d} completes the proof. 
\end{proof} 

\begin{remark}
It should be noted that polynomial correspondences for which the equation $f_{d} (x, y) = 0$ can be rewritten as $\Phi (x) = \Psi (y)$, as mentioned in Theorem \ref{Bullett,separation,for,d}, simplify the process of tracing the orbits arising from the correspondence. The orbit of any point $x_{0} \in \mathbb{\widehat{C}}_{x}$ and $y_{0} \in \mathbb{\widehat{C}}_{y}$ are determined by 
\[ x_{0} \xrightarrow{f_{d}} x_{1} \in \Psi^{-1} \left( \Phi \left( x_{0} \right) \right) \xrightarrow{f_{d}} x_{2} \in \Psi^{-1} \left( \Phi \left( x_{1} \right) \right) \cdots\ \ \ \text{or}\ \ \ y_{0} \xrightarrow{f_{d}} y_{1} \in \Phi^{-1} \left( \Psi \left( y_{0} \right) \right) \xrightarrow{f_{d}} y_{2} \in \Phi^{-1} \left( \Psi \left( y_{1} \right) \right) \cdots, \] 
depending on the direction we are interested in. In other words, the correspondence behaves like a multi-valued map $\left( \Psi^{-1} \circ \Phi \right)$ or $\left( \Phi^{-1} \circ \Psi \right)$, locally. 
\end{remark}

\section{Certain interesting symmetries in $f_{d}$} 
\label{pfsec2} 

In this section, we focus on certain interesting symmetries that one may expect from a map of $d$-tuples $f_{d}$. 

\begin{proposition} 
\label{symmetry} 
Let $f_{d} = 0$ be a map of $d$-tuples that can not be written as $\displaystyle{f_{d} (x, y) = \big[ g_{k} (x, y) \big]^{\frac{d}{k}}}$ where $k | d$ and $k \ne d$ with $g_{k} \in \mathbb{C}[x, y]$ being a map of $k$-tuples. 
\begin{enumerate} 
\item The following statements are equivalent. 
\label{symm} 
\begin{enumerate} 
\item $f_{d} (x, y) = 0$ if and ony if $f_{d} (y, x) = 0$. 
\item There exists a factorisation of $f_{d} = 0$, as mentioned in Equation \eqref{phipsi} where the complex coefficients $\kappa_{i}, \lambda_{i}, \mu_{i}$ and $\nu_{i};\; 0 \le i \le d$ satisfy $\kappa_{i} \nu_{j} - \lambda_{i} \mu_{j}  = \pm  (\kappa_{j} \nu_{i} - \lambda_{j} \mu_{i})$. 
\item The matrix of coefficients of $f_{d}$ is symmetric or antisymmetric. 
\end{enumerate} 
\item The following statements are equivalent. 
\label{timepres} 
\begin{enumerate} 
\item $f_{d} (x, y) = 0$ if and ony if $f_{d} (\overline{x}, \overline{y}) = 0$. 
\item There exists a factorisation of $f_{d} = 0$, as mentioned in Equation \eqref{phipsi} where the coefficients $\kappa_{i}, \lambda_{i}, \mu_{i}$ and $\nu_{i};\; 0 \le i \le d$ are real. 
\item The matrix of coefficients of $f_{d}$ is real (upto multiplication by a constant). 
\end{enumerate} 
\item The following statements are equivalent. 
\label{timerev} 
\begin{enumerate} 
\item $f_{d} (x, y) = 0$ if and ony if $f_{d} (\overline{y}, \overline{x}) = 0$. 
\item There exists a factorisation of $f_{d} = 0$, as mentioned in Equation \eqref{phipsi} where the coefficients $\kappa_{i}, \lambda_{i}, \mu_{i}$ and $\nu_{i};\; 0 \le i \le d$ satisfy $\kappa_{i} = \overline{\mu_{i}}$ and $\lambda_{i} = \overline{\nu_{i}}$. 
\item The matrix of coefficients of $f_{d}$ is skew-Hermitian (upto multiplication by a constant). 
\end{enumerate} 
\end{enumerate} 
\end{proposition} 

We now state and prove a lemma that will be useful in proving the first set of equivalent statements in Proposition \ref{symmetry}.

\begin{lemma} 
\label{symmlemma}
Let $f_{d} (x, y) = 0$ be a map of $d$-tuples that can be rewritten as $\Phi (x) = \Psi (y)$.  Suppose there exists a distinct pair of $d$-tuples namely $\big( x_{1}, \cdots, x_{d} \big)$ and $\big( y_{1}, \cdots, y_{d} \big)$ with $x_{i}, y_{j} \in \widehat{\mathbb{C}}$ arranged as demanded by Remark \ref{rmodtuple}, such that $f_{d} (x_{i}, y_{j}) = 0$ if and ony if $f_{d} (y_{j}, x_{i}) = 0$ for $1 \le i, j \le d$, then, $f_{d} (x, y) = 0$ can also be rewritten as $\Phi (y) = \Psi (x)$. 
\end{lemma} 

\begin{proof} 
In order to prove the Lemma, we merely prove that the coefficient of $x^{d} y^{d - 1}$ in the expression of $f_{d}$ agrees with the coefficient of $x^{d - 1} y^{d}$. Using analogous techniques, we can then prove that the coefficient of $x^{i} y^{j}$ in the expression of $f_{d}$ is the same as the coefficient of $x^{j} y^{i}$, that completes the proof. 

Since $f_{d} (x, y) = 0$ is given to be a map of $d$-tuples, we know by Theorem \ref{Bullett,separation,for,d} that the equation can be rewritten as $\Phi (x) = \Psi (y)$, where $\Phi$ and $\Psi$ are fractional degree $d$ functions, that look like 
\[ \Phi (x)\ \ =\ \ \frac{\kappa_{d} x^{d}\, +\, \cdots\, +\, \kappa_{1} x\, +\, \kappa_{0}}{\lambda_{d} x^{d}\, +\, \cdots\, +\, \lambda_{1} x\, +\, \lambda_{0}} \quad \quad \text{and} \quad \quad \Psi (y)\ \ =\ \ \frac{\mu_{d} y^{d}\, +\, \cdots\, +\, \mu_{1} y\, +\, \mu_{0}}{\nu_{d} y^{d}\, +\, \cdots\, +\, \nu_{1} y\, +\, \nu_{0}}, \] 
with complex coefficients. 

{\bf The finite case:}  In this case, we consider all $x_{i}$'s, $y_{j}$'s and $z_{k}$'s to be finite, as explained below. Given that 
\begin{eqnarray*} 
\big( x_{1}, x_{2}, \cdots, x_{d} \big)_{\widehat{\mathbb{C}}_{x}} &\mathrel{\mathop{\rightleftarrows}^{f_{d}}} & \big( y_{1}, y_{2}, \cdots, y_{d} \big)_{\widehat{\mathbb{C}}_{y}}\ \ \ \text{and} \\ 
\big( y_{1}, y_{2}, \cdots, y_{d} \big)_{\widehat{\mathbb{C}}_{x}} & \mathrel{\mathop{\rightleftarrows}^{f_{d}}} & \big( x_{1}, x_{2}, \cdots, x_{d} \big)_{\widehat{\mathbb{C}}_{y}}, 
\end{eqnarray*} 
we consider distinct values $z_{1}$ and $z_{2}$ in $\mathbb{C}$ such that 
\[ \Phi (x_{i})\ =\ z_{1}\ =\ \Psi (y_{j})\ \ \ \ \ \ \text{and}\ \ \ \ \ \ \Phi (y_{j})\ =\ z_{2}\ =\ \Psi (x_{i})\ \ \ \ \ \forall 1 \le i, j \le d. \] 
In other words, $\{ x_{1}, \cdots, x_{d} \}$ is the solution set of the algebraic equations $\Phi (x) - z_{1} = 0$ and $\Psi (x) - z_{2} = 0$, where points are repeated according to their multiplicities. Employing the expressions that we have for $\Phi$ and $\Psi$, as written above and considering the sum of these roots from these algebraic equations, we obtain 
\[ \frac{\kappa_{d - 1} - z_{1} \lambda_{d - 1}}{\kappa_{d} - z_{1} \lambda_{d}}\ \ =\ \  \frac{\mu_{d - 1} - z_{2} \nu_{d - 1}}{\mu_{d} - z_{2} \nu_{d}}. \] 
Similarly, concentrating on the solution set, namely $\{ y_{1}, \cdots, y_{d} \}$ of the algebraic equations $\Psi (y) - z_{1} = 0$ and $\Phi (y) - z_{2} = 0$, where again the points are repeated according to their multiplicities, we obtain the sum of these roots as 
\[ \frac{\kappa_{d - 1} - z_{2} \lambda_{d - 1}}{\kappa_{d} - z_{2} \lambda_{d}}\ \ =\ \  \frac{\mu_{d - 1} - z_{1} \nu_{d - 1}}{\mu_{d} - z_{1} \nu_{d}}. \] 
From the above two equations, we get, 
\[ z_{1} \left( \lambda_{d} \mu_{d - 1} - \lambda_{d - 1} \mu_{d} \right) + z_{2} \left( \kappa_{d} \nu_{d - 1} - \kappa_{d - 1} \nu_{d} \right)\ =\ z_{1} \left( \kappa_{d} \nu_{d - 1} - \kappa_{d - 1} \nu_{d} \right) + z_{2} \left( \lambda_{d} \mu_{d - 1} - \lambda_{d - 1} \mu_{d} \right), \] 
that gives 
\[ z_{1} \left( \kappa_{d} \nu_{d - 1} - \kappa_{d - 1} \nu_{d} - \lambda_{d} \mu_{d - 1} + \lambda_{d - 1} \mu_{d} \right)\ =\ z_{2} \left( \kappa_{d} \nu_{d - 1} - \kappa_{d - 1} \nu_{d} - \lambda_{d} \mu_{d - 1} + \lambda_{d - 1} \mu_{d} \right). \] 
Since $z_{1} \ne z_{2}$, we have 
\[ \kappa_{d} \nu_{d - 1} - \lambda_{d} \mu_{d - 1}\ \ =\ \ \kappa_{d - 1} \nu_{d} - \lambda_{d - 1} \mu_{d}, \] 
which proves that the coefficients of $x^{d} y^{d - 1}$ and $x^{d - 1} y^{d}$ are equal in the equation $f_{d} (x, y) = 0$. 

{\bf The infinite case:} This has various subcases, one or several of $x_{i}$'s or $y_{j}$'s can be infinite (depending upon the multiplicity with which $\infty$ appears as a zero of the appropriate equation) or one of the $z_{k}$'s is infinite. However, note that $\infty$ can not appear in the enlisting of both $x_{i}$'s and $y_{j}$'s simultaneously. Owing to symmetric reasonings, we will only discuss $x_{i}$'s being infinite, here. 

{\bf Subcase - 1: When $x_{i}$'s and $y_{j}$'s are all finite, $z_{1} < \infty$ but $z_{2} = \infty$.} In this case, we consider the equations 
\[ \Phi (x_{i})\ =\ z_{1}\ =\ \Psi (y_{j})\ \ \ \ \ \ \text{and}\ \ \ \ \ \ \Phi (y_{j})\ =\ \infty\ =\ \Psi (x_{i})\ \ \ \ \ \forall 1 \le i, j \le d. \] 
Thus, $\{ x_{1}, \cdots, x_{d} \}$ is the solution set of the algebraic equations $\sum\limits_{i = 0}^{d} \left( \kappa_{i} - z_{1} \lambda_{i} \right) x^{i} = 0$ and $\sum\limits_{i = 0}^{d} \nu_{i} x^{i} = 0$, where the points are repeated according to their multiplicities. Focussing again on the sum of all these $x_{i}$'s yields 
\[ z_{1}\ =\ \frac{\kappa_{d} \nu_{d - 1} - \kappa_{d - 1} \nu_{d}}{\lambda_{d} \nu_{d - 1} - \lambda_{d - 1} \nu_{d}}. \] 
When we do an analogous exercise with $\{ y_{1}, \cdots, y_{d} \}$ being the solution set of the algebraic equations $\sum\limits_{i = 0}^{d} \left( \mu_{i} - z_{1} \nu_{i} \right) y^{i} = 0$ and $\sum\limits_{i = 0}^{d} \lambda_{i} y^{i} = 0$, where again the points are repeated according to their multiplicities, we obtain 
\[ z_{1}\ =\ \frac{\lambda_{d} \mu_{d - 1} - \lambda_{d - 1} \mu_{d}}{\lambda_{d} \nu_{d - 1} - \lambda_{d - 1} \nu_{d}}. \] 
Since the denominators in the expressions of $z_{1}$ are equal, their numerators should also be equal that results in 
\[ \kappa_{d} \nu_{d - 1} - \lambda_{d} \mu_{d - 1}\ \ =\ \ \kappa_{d - 1} \nu_{d} - \lambda_{d - 1} \mu_{d}. \] 

{\bf Subcase - 2: When only $(d - m)$ of the $x_{i}$'s are finite, the remainder of the $x_{i}$'s being $\infty$ while all of the $y_{j}$'s and $z_{k}$'s are finite.} Owing to $m$ many infinities in the collection of $x_{i}$'s and the equations $\Phi (x_{i}) - z_{1} = 0$ and $\Psi (x_{i}) - z_{2} = 0$, we get $z_{1} = \dfrac{\kappa_{i}}{\lambda_{i}}$ and $z_{2} = \dfrac{\mu_{i}}{\nu_{i}}$ for $i = d , d - 1, \cdots, d - m$. Using the sum of the roots of the algebraic equations $\Psi (y) - z_{1} = 0$ and $\Phi (y) - z_{2} = 0$ and plugging the values of $z_{1}$ and $z_{2}$, as obtained above, we get 
\[ \kappa_{d} \nu_{d - 1} - \lambda_{d} \mu_{d - 1}\ \ =\ \ \kappa_{d - 1} \nu_{d} - \lambda_{d - 1} \mu_{d}. \] 

{\bf Subcase - 3: When only $(d - m)$ of the $x_{i}$'s are finite, the remainder of the $x_{i}$'s being $\infty$ while all of the $y_{j}$'s are finite, $z_{1} < \infty$ and $z_{2} = \infty$.} We reach the same conclusion in this final and exhaustive case too, working analogously. 
\end{proof} 

\begin{remark}
\label{symmremark}
Let $f_{d} (x, y) = 0$ be a map of $d$-tuples that satisfies the condition  $f_{d} (x, y) = 0$ if and only if $f_{d} (y, x) = 0$. Suppose there does not exist any distinct pair of $d$-tuples, $\big( x_{1}, \cdots, x_{d} \big)$ and $\big( y_{1}, \cdots, y_{d} \big)$ with $x_{i}, y_{j} \in \widehat{\mathbb{C}}$ arranged as demanded by Remark \eqref{rmodtuple}, such that $f_{d} (x_{i}, y_{j}) = 0$, then $f_{d} (x, y) = 0$ can be written as $\Phi (x) = \Psi (y)$, with $\Phi = \Psi$
\end{remark}

\begin{proof}(of the first set of equivalent statements in Proposition \ref{symmetry}) \\
The hypothesis in the statement of the theorem divides the proof into two cases given by the mutually exclusive conditions, as written in statements of Lemma \ref{symmlemma} and Remark \ref{symmremark}. The proof of Case $1$ rests entirely on Lemma \ref{symmlemma}. It is quite obvious from the proof of Lemma \ref{symmlemma} that the matrix of coefficients is symmetric and that the coefficients, as mentioned in the equation $\Phi (x) = \Psi (y)$ satisfy 
\[ \kappa_{i} \nu_{j} + \lambda_{j} \mu_{i}\ =\ \kappa_{j} \nu_{i} + \lambda_{i} \mu_{j},\ \ \ \ \text{for}\ \ \ 0 \le i, j \le d. \] 

Now, suppose $f_d$ holds the condition mentioned in Remark \ref{symmremark}. Then $f_{d} (x, y) = 0$ can be written in the form 
\[ \frac{\kappa_{d} x^{d}\, +\, \cdots\, +\, \kappa_{1} x\, +\, \kappa_{0}}{\lambda_{d} x^{d}\, +\, \cdots\, +\, \lambda_{1} x\, +\, \lambda_{0}} \quad \quad =\ \ \frac{\mu_{d} y^{d}\, +\, \cdots\, +\, \mu_{1} y\, +\, \mu_{0}}{\nu_{d} y^{d}\, +\, \cdots\, +\, \nu_{1} y\, +\, \nu_{0}}, \] 
with $\kappa_{i} =  \mu_{i} \ \text{and}\ \lambda_{i} =  \nu_{i} $ for $0 \le i \le d$. This proves condition $(b)$ and from that it is clear that the matrix of coefficients is anti-symmetric.

Finally, if the matrix of coefficients of $f_d=0$, is symmetric or anti-symmetric, then condition $(a)$ obviously holds. 
\end{proof} 

\noindent 
Of the equivalent statements in the second set in Proposition \ref{symmetry}, we first focus on the nontrivial implication, namely $(a) \Longrightarrow (c)$. 

\begin{proof} 
$(a) \Longrightarrow (c)$: Since $f_{d}$ is a map of $d$-tuples, we have 
\[ \big( x_{1}, x_{2}, \cdots, x_{d} \big)_{\widehat{\mathbb{C}}_{x}}\ \ \mathrel{\mathop{\rightleftarrows}^{f_{d}}}\ \ \big( y_{1}, y_{2}, \cdots, y_{d} \big)_{\widehat{\mathbb{C}}_{y}}, \] 
as the solution set for $f_{d} (x, y) = 0$. Owing to our assumption, we further have $f_{d} (\overline{x}, \overline{y}) = 0$, {\it i.e.}, 
\[ \big( \overline{x_{1}}, \overline{x_{2}}, \cdots, \overline{x_{d}} \big)_{\widehat{\mathbb{C}}_{x}}\ \ \mathrel{\mathop{\rightleftarrows}^{f_{d}}}\ \ \big( \overline{y_{1}}, \overline{y_{2}}, \cdots, \overline{y_{d}} \big)_{\widehat{\mathbb{C}}_{y}}. \] 
Taking the complex conjugate of the latter equation, we obtain 
\[ \big( x_{1}, x_{2}, \cdots, x_{d} \big)_{\widehat{\mathbb{C}}_{x}}\ \ \mathrel{\mathop{\rightleftarrows}^{\overline{f_{d}}}}\ \ \big( y_{1}, y_{2}, \cdots, y_{d} \big)_{\widehat{\mathbb{C}}_{y}}, \] 
where $\overline{f_{d}}$ denotes the correspondence mentioned in Equation \eqref{fdnozeta}, with every coefficient $A_{i, j}$ replaced with $\overline{A_{i, j}}$. Thus, $f_{d} (x, y) = 0$ and $\overline{f_{d}} (x, y) = 0$ represent the same polynomial correspondence, meaning one is a constant multiple of the other. By comparing the coefficients, we obtain that $A_{i, j} = M \overline{A_{i, j}}$ for every $0 \le i, j \le d$, where $M \in \mathbb{C}$. Thus, we have ${\rm Re}(A_{i, j}) = \rho {\rm Im} (A_{i, j})$ for every $0 \le i, j \le d$ for some $\rho \in \mathbb{R}$, since for any two points $z, w \in \mathbb{C}$ satisfying $z \overline{w} = w \overline{z}$, we have ${\rm Re}(z) {\rm Im}(w) = {\rm Re}(w) {\rm Im}(z)$. Hence, writing $A_{i, j} = \rho a_{i, j} +  i a_{i, j}$, we obtain an alternate representation of the correspondence given by 
\begin{eqnarray*} 
f_{d}(x, y) & =\ \ (\rho + i) & \big[ a_{(d,\, d)} x^{d} y^{d}\; +\; a_{(d,\, d - 1)} x^{d} y^{d - 1}\; +\; \cdots\; +\; a_{(d,\, 1)} x^{d} y\; +\; a_{(d,\, 0)} x^{d} \\ 
& & +\; a_{(d - 1,\, d)} x^{d - 1} y^{d}\; +\; \cdots\; +\; a_{(d - 1,\, 0)} x^{d - 1} \\ 
& & +\; \cdots \\ 
& & +\; a_{(0,\, d)} y^{d}\; +\; \cdots\; +\; a_{(0,\, 0)} \big],
\end{eqnarray*}
where every coefficient $a_{i, j} \in \mathbb{R}$, proving statement $(c)$. 

$(c) \Longrightarrow (b)$: Owing to the hypothesis that we have regarding $f_{d}$ in the statement of the theorem and the assumption that the matrix of coefficients of $f_{d}$ is real, we have from Theorem \ref{Bullett,for,d} that the matrix of coefficients of $f_{d}$ has rank $2$. Hence, the row vector space is a real vector space of dimension $2$. Therefore, following the construction in the proof of Theorem \ref{Bullett,for,d}, we conclude that there exists a factorisation of the equation $f_{d} = 0$, as proved in Theorem \ref{Bullett,separation,for,d}, given by $\Phi (x) = \Psi (y)$ where the coefficients $\kappa_{i},\ \lambda_{i},\, \mu_{i}$ and $\nu_{i}$ are real. 

$(b) \Longrightarrow (a)$: This is trivial. 
\end{proof} 

Among the equivalent statements in the third set in Proposition \ref{symmetry}, we first focus on the nontrivial implication, namely $(a) \Longrightarrow (b)$. 

\begin{proof} 
$(a) \Longrightarrow (b)$: Since $f_{d}$ is a map of $d$-tuples, by Theorem \ref{Bullett,separation,for,d} we obtain a factorisation of $f_{d} (x, y) = 0$ as $\Phi (x) = \Psi (y)$. Then, making use of the hypothesis, we obtain a factorisation of $f_{d} (\overline{y}, \overline{x}) = 0$ as $\Phi (\overline{y}) = \Psi (\overline{x})$. Taking the complex conjugate of the last equation gives us $\overline{\Phi (\overline{y})} = \overline{\Psi (\overline{x})}$, which is alternate representation of the original correspondence $f_{d} (x, y) = 0$. Thus, this only means that there exists a M\"{o}bius map (of determinant $1$, without loss of generality) $M = \begin{pmatrix} A & B \\ C & D \end{pmatrix}$ such that $\overline{\Psi} (x) = M \circ \Phi (x)$. Here, $\overline{\Psi}$ represents the fractional degree $d$ rational map, as written in Equation \eqref{phipsi} where every coefficient in the expression of $\Psi$ is complex conjugated. Thus, 
\begin{equation} 
\label{psibarxone} 
\overline{\Psi} (x)\ \ =\ \ M \circ \Psi (y). 
\end{equation} 
Owing to our hypothesis as in statement $(a)$, we also have $\overline{\Psi} (\overline{y}) = M \circ \Psi (\overline{x})$. Taking the complex conjugate of the above equation, we obtain $\Psi (y) = \overline{M} \circ \overline{\Psi} (x)$. Thus, 
\begin{equation} 
\label{psibarxtwo} 
\overline{\Psi} (x)\ \ =\ \ \left( \overline{M} \right)^{-1} \circ \Psi (y). 
\end{equation} 
From Equations \eqref{psibarxone} and \eqref{psibarxtwo}, we obtain $M = \left( \overline{M} \right)^{-1}$, that results in the matrix entries $B$ and $C$ being purely imaginary and $A = \overline{D}$. Such a matrix $M$ can be expressed as $M = \left( \overline{N} \right)^{-1} N$ for some matrix $N$ with complex entries. Plugging in this expression for $M$ in Equation \eqref{psibarxone}, we get 
\[ N \circ \Psi (y)\ \ =\ \ \overline{N} \circ \overline{\Psi} (x), \] 
establishing $(b)$. 

$(b) \Longrightarrow (c)$: Since $\kappa_{i} = \overline{\mu_{i}}$ and $\lambda_{i} = \overline{\nu_{i}}$, we obtain that the complex conjugate of the coefficient of $x^{i} y^{j}$ to be equal to the negative of the coefficient of $x^{j} y^{i}$, making the corresponding matrix of coefficients to be skew-Hermitian. 

$(c) \Longrightarrow (a)$: This is trivial. 
\end{proof} 

\section{Examples}
\label{exsec} 

In this concluding section, we provide examples illustrating the main theorems and the symmetries discussed so far in this paper. 

\begin{example} 
Consider 
\begin{displaymath} 
\begin{array}{c r c l} 
& 5 x^{5} y^{5}\; +\; 5 x^{5} y^{4}\; +\; 10 x^{5} y^{3}\; +\; 7 x^{5} y^{2}\; + (1 + 6i) x^{5} y\; +\; 11 x^{5} & & \nonumber \\ 
+ & 7 x^{4} y^{5}\; +\; 4 x^{4} y^{4}\; +\; 8 x^{4} y^{3}\; +\; 11 x^{4} y^{2}\; +\; (2 + 3i) x^{4} y\; +\; 7 x^{4} & & \nonumber \\ 
+ & (9 - i) x^{3} y^{5}\; +\; (3 - 2i) x^{3} y^{4}\; +\; (6 - 4i) x^{3} y^{3}\; +\; (15 - i) x^{3} y^{2}\; +\; 6 x^{3} y\; +\; (3 - 5i) x^{3} & & \nonumber \\ 
+ & 15 x^{2} y^{5}\; +\; 10 x^{2} y^{4}\; +\; 20 x^{2} y^{3}\; +\; 23 x^{2} y^{2}\; +\; (4 + 9i) x^{2} y\; +\; 19 x^{2} & & \nonumber \\ 
+ & 16 x y^{5}\; +\; 7 x y^{4}\; +\; 14 x y^{3}\; +\; 26 x y^{2}\; +\; (5 + 3i) x y\; +\; 10 x & & \nonumber \\ 
+ & 7 y^{5}\; +\; 9 y^{4}\; +\; 18 y^{3}\; +\; 9 y^{2}\; +\; (1 + 12i) y\; +\; 21 & = & 0. 
\end{array} 
\end{displaymath} 
This correspondence can be factorised as 
\[ \frac{x^{5}\, +\, 2 x^{4}\, +\, 3 x^{3}\, +\, 4 x^{2}\, +\, 5 x\, +\, 1}{2x^{5}\, +\, x^{4}\, -\, i x^{3}\, +\, 3 x^{2}\, +\, x\, +\, 4}\ \ =\ \ -\; \frac{y^{5}\, +\, 2 y^{4}\, +\, 4 y^{3}\, +\, y^{2}\, +\, 3i y\, +\, 5}{3 y^{5}\, +\, y^{4}\, +\, 2 y^{3}\, +\, 5 y^{2}\, +\, y\, +\, 1}. \] 
Thus, by Theorem \ref{Bullett,separation,for,d}, this is a map of $5$-tuples. Further, the matrix of coefficients of the correspondence, as stated in Theorem \ref{Bullett,for,d} has rank $2$. 
%, {\it i.e.}, 
%\[ {\rm rank} \begin{pmatrix} 
%5 & 5 & 10 & 7 & 1 + 6i & 11 \\ 
%7 & 4 & 8 & 11 & 2 + 3i & 7 \\ 
%9 - i & 3 - 2i & 6 - 4i & 15 - i & 6 & 3 - 5i \\ 
%15 & 10 & 20 & 23 & 4 + 9i & 19 \\ 
%16 & 7 & 14 & 26 & 5 + 3i & 10 \\ 
%7 & 9 & 18 & 9 & 1 + 12i & 21 
%\end{pmatrix}\ \ =\ \ 2. \] 
\end{example} 

\begin{example} 
\begin{enumerate} 
\item[(a)] The correspondence 
\begin{eqnarray*} 
x^{3} y^{3}\; +\; 3 x^{3} y^{2}\; +\; 3 x^{3} y\; +\; x^{3}\; +\; 3 x^{2} y^{3}\; +\; 12 x^{2} y^{2}\; +\; 15 x^{2} y\; +\; 6 x^{2} & & \\ 
+ 3 x y^{3}\; +\; 15 x y^{2}\; +\; 24 x y\; +\; 12 x\; +\; y^{3}\; +\; 6 y^{2}\; +\; 12 y\; +\; 8 & = & 0, 
\end{eqnarray*} 
can be written as $\left( x y + x + y + 2 \right)^{3}\ =\ 0$, making this correspondence a map of $3$-tuples. 
\item[(b)] The correspondence 
\begin{eqnarray*} 
x^{4} y^{4}\; +\; 2 x^{4} y^{2}\; +\; x^{4}\; +\; 2 x^{3} y^{4}\; +\; 4 x^{3} y^{2}\; +\; 2 x^{3}\; +\; 3 x^{2} y^{4}\; +\; 8 x^{2} y^{2} & & \\ 
+ 5 x^{2}\; +\; 2 x y^{4}\; +\; 6 x y^{2}\; +\; 4 x\; +\; y^{4}\; +\; 4 y^{2}\; +\; 4 & = & 0, 
\end{eqnarray*} 
can be written as $\left( x^{2} y^{2} + x^{2} + x y^{2} + x + y^{2} + 2 \right)^{2}\ =\ 0$, making this correspondence a map of $4$-tuples; however with the multiplicity of each of the roots of the quadratic correspondence inside the bracket, being doubled for the quartic correspondence. 
\end{enumerate} 
\end{example} 

\begin{example} 
Consider the map of $3$-tuples given by 
\[ \frac{x^{3}\, -\, 6 x^{2}\, +\, 11 x\, -\, 6}{- x^{3}\, -\, 8 x^{2}\, +\, 31 x\, +\, 10}\ \ =\ \ \frac{2 y^{3}\, +\, 2 y^{2}\, -\, 20 y\, -\, 16}{y^{3}\, +\, 8 y^{2}\, -\, 31 y\, -\, 10}, \] 
that corresponds between the triples $(1, 2, 3)$ and $(-1, -2, 4)$ {\it via} $z_{1} = 0$ and $z_{2} = 1$, where we follow the notations for $z_{k}$'s, as in the proof of Lemma \ref{symmlemma}. Observe that this correspondence can also be expressed as 
\begin{eqnarray*} 
3 x^{3} y^{3}\; +\; 10 x^{3} y^{2}\; -\; 51 x^{3} y\; -\; 26 x^{3}\; +\; 10 x^{2} y^{3}\; -\; 32 x^{2} y^{2}\; +\; 26 x^{2} y\; -\; 68 x^{2} & & \\ 
- 51 x y^{3}\; +\; 26 x y^{2}\; +\; 279 x y\; +\; 386 x\; -\; 26 y^{3}\; -\; 68 y^{2}\; +\; 386 y\; +\; 220 & = & 0, 
\end{eqnarray*} 
wherein the coefficient of $x^{i} y^{j}$ is the same as the coefficient of $x^{j} y^{i}$, for all $0 \le i, j \le 3$, leading to a symmetric matrix of coefficients, as required in the first set of equivalent statements in Proposition \ref{symmetry}. 
\end{example} 

\begin{example} 
Consider the map of $4$-tuples given by 
\[ \frac{2 x^{4}\, +\, 3 x^{3}\, +\, 10 x^{2}\, +\, x\, +\, 7}{x^{4}\, +\, 5 x^{3}\, +\, x^{2}\, +\, 4 x\, +\, 1}\ \ =\ \ -\; \frac{y^{4}\, +\, 2 y^{3}\, +\, 5 y^{2}\, +\, 6 y\, +\, 4}{y^{4}\, +\, 7 y^{3}\, +\, 2 y^{2}\, +\, y\, +\, 8}, \] 
where all the coefficients in the above mentioned fractional quartic representation are real. Note that the correspondence when written as $f_{4} (x,\, y) = 0$ has all real coefficients. Further, taking the complex conjugate, we observe that this correspondence also satisfies condition $(a)$ in the second set of equivalent statements in Proposition \ref{symmetry}.  
\end{example} 

\begin{example} 
Consider the following rank $2$, skew-Hermitian matrix of coefficients of a cubic correspondence given by 
\[ \begin{pmatrix} 
28i & 6i - 25 & 19 i - 7 & - 31 \\ 
25 + 6i & 12 i & 17 + 5i & - 15 + 4i \\ 
7 + 19 i & - 17 + 5i & 12 i & - 2i - 28 \\ 
31 & 15 + 4i & 28 - 2i & - 20 i
\end{pmatrix}. \] 
The correspondence $f_{3}$ formed from the above matrix of coefficients is a map of $3$-tuples, that satisfies $f_{3}(x, y) = 0 $ if and ony if $f_{3}(\overline{y}, \overline{x}) = 0$. Then, $f_{3}(x, y) = 0$ can also be rewritten as 
\[ \frac{2i x^{3}\, +\, 3 x^{2}\, +\, (1 + i) x\, +\, 5}{7 x^{3}\, +\, (3 - 2i) x^{2}\, +\, 6 x\, +\, 2i}\ \ =\ \ \frac{- 2i y^{3}\, +\, 3 y^{2}\, +\, (1 - i) y\, +\, 5}{7 y^{3}\, +\, (3 + 2i) y^{2}\, +\, 6 y\, -\, 2i}, \] 
illustrating the third set of equivalent statements in Proposition \ref{symmetry}. 
\end{example} 
\bigskip 

%{\bf Conflict of Interests statement:} The authors declare that they have no relevant financial or non-financial interests to disclose.
\bigskip

%{\bf Data Availability statement:} Apart from the references mentioned below, there is no associated data that was used for the preparation of this manuscript. 

\bigskip \bigskip  

\end{document}